\documentclass[letterpaper]{article} 
\usepackage{aaai24}  
\usepackage{times}  
\usepackage{helvet}  
\usepackage{courier}  
\usepackage[hyphens]{url}  
\usepackage{graphicx} 
\usepackage{natbib}  
\usepackage{caption} 
\usepackage{amsmath}
\usepackage{amsthm}
\usepackage{amsfonts}
\usepackage{amssymb}
\newtheorem{theorem}{Theorem}
\usepackage[pdftex,dvipsnames]{xcolor}  
\usepackage{algorithm}
\usepackage{algorithmic}

\usepackage{newfloat}
\usepackage{listings}
\DeclareCaptionStyle{ruled}{labelfont=normalfont,labelsep=colon,strut=off} 
\floatstyle{ruled}
\newfloat{listing}{tb}{lst}{}
\floatname{listing}{Listing}
\title{Optimizing ADMM and Over-Relaxed ADMM  Parameters for \\Linear Quadratic Problems}
\author{
   Jintao Song\textsuperscript{\rm 1,\rm2}, Wenqi Lu\textsuperscript{\rm3,\rm4}, Yunwen Lei\textsuperscript{\rm 5}, Yuchao Tang\textsuperscript{\rm 6}, Zhenkuan Pan\textsuperscript{\rm 2}, Jinming Duan\textsuperscript{\rm 1{$\dagger$}}\\
  }
\affiliations{
    \textsuperscript{\rm 1} School of Computer Science, University of Birmingham, UK\\
    \textsuperscript{\rm 2} College of Computer Science and Technology, Qingdao University, China\\

    \textsuperscript{\rm 3} Department of Computing and Mathematics, Manchester Metropolitan University, UK\\
    \textsuperscript{\rm 4} Centre for Computational Science and Mathematical Modelling, Coventry University, UK\\
    \textsuperscript{\rm 5} Department of Mathematics, University of Hong Kong, HK\\
    \textsuperscript{\rm 6} School of Mathematics and Information Science, Guangzhou University, China\\
\vspace{13pt}
}

\usepackage{bibentry}

\begin{document}

\maketitle

\begin{abstract}
The Alternating Direction Method of Multipliers (ADMM) has gained significant attention across a broad spectrum of machine learning applications. Incorporating the over-relaxation technique shows potential for enhancing the convergence rate of ADMM. However, determining optimal algorithmic parameters, including both the associated penalty and relaxation parameters, often relies on empirical approaches tailored to specific problem domains and contextual scenarios. Incorrect parameter selection can significantly hinder ADMM's convergence rate. To address this challenge, in this paper we first propose a general approach to optimize the value of penalty parameter, followed by a novel closed-form formula to compute the optimal relaxation parameter in the context of linear quadratic problems (LQPs). We then experimentally validate our parameter selection methods through random instantiations and diverse imaging applications, encompassing diffeomorphic image registration, image deblurring, and MRI reconstruction.
\end{abstract}

\section{Introduction}

ADMM is a versatile algorithm with applications spanning various domains, including compressed sensing \cite{hou2022truncated,liu2023distributed}, image processing \cite{chan2016plug,yazaki2019interpolation}, and machine learning \cite{li2022robust,zhou2023federated}. Although introduced in the 1970s for optimization, its roots can be traced back to the 1950s as a method to solve elliptic and parabolic partial difference equations \cite{boyd2011distributed}. ADMM leverages the convergence strengths of the method of multipliers and the decomposability property of dual ascent. It is particularly useful in addressing convex optimization of considerable scale, beyond the capacity of conventional solvers. The ongoing research and outstanding algorithmic performance have significantly contributed to its widespread adoption, highlighting the growing importance of exploring its theoretical properties, particularly regarding parameter selection \cite{ghadimi2014optimal,wang2019admm}.

ADMM, from a technical viewpoint, decomposes complex optimization problems into manageable sub-problems, often solvable using point-wise, closed-form solvers \cite{candes2011robust,lu2016implementation,thorley2021nesterov,jia2021learning,duan2023arbitrary}. It proceeds by iteratively updating these sub-problems alternately until a solution meeting the original problem's objectives and constraints is attained. Within ADMM, the augmented Lagrange function incorporates penalty terms associated with the constraints. The penalty parameters determine the strength of these penalty terms. As highlighted in \cite{deng2016global}, the convergence rate of ADMM is directly impacted by these penalty parameters. The optimal selection of such parameters can significantly enhance the algorithm's convergence rate. However, the lack of a universal method to compute these parameters optimally remains a challenge.

The convergence rate of ADMM can be further accelerated by leveraging information from prior iterations during the computation of subsequent iterations. Such a technique is known as over-relaxation and often used in conjunction with ADMM \cite{de1986relaxed,zhang2020privacy}. Numerous research endeavors have been devoted to defining appropriate values for the resultant relaxation parameter. Notably, in the study conducted by \cite{eckstein1994parallel}, the authors proposed a widely acknowledged empirical range of values, typically falling within $[1.5, 1.8]$, which however is not always the case according to our findings in this paper. Despite a multitude of papers presenting specific guidelines for selecting this parameter, many real-world application papers \cite{stellato2020osqp,duan2023arbitrary} still resort to empirically determined values. This reliance on empirical choices is due to the absence of a straightforward and efficient method that can promptly and optimally determine this relaxation parameter.

The objective of this paper is to introduce novel methods for the selection of optimal parameters within both ADMM and over-relaxed ADMM. As an example, we focus on linear quadratic problems (LQPs), particularly with applications tailored to image processing. The theories developed in this paper could offer valuable insights for addressing other non-quadratic problems, such as non-smooth $L_1$ optimization. More specifically, we have identified four key contributions of this paper, summarized as follows:

\begin{itemize}
\item We perform a comprehensive convergence analysis of the ADMM algorithm as applied to LQPs, effectively demonstrating its unconditional convergence within the context of LQPs. This is achieved by initially converting the ADMM iterations into its fixed-point iterations, which facilitates the derivation of the iteration matrix. Subsequently, we theoretically show that the spectral radius of the iteration matrix is bounded by $1$, regardless of the value of the penalty parameter.
\item We propose a general optimization method for the selection of the optimal penalty parameter in ADMM. We achieve this by utilizing numerical gradient descent to minimize the spectral radius of the iteration matrix. Moreover, in specific scenarios like image deblurring and MRI reconstruction, we show the existence of an closed-form solution for accurately determining the optimal penalty parameter within ADMM.
\item We establish, for the first time, the existence of an closed-form solution for determining the relaxation parameter in over-relaxed ADMM. We find that for any arbitrary value of the penalty parameter, there exists a corresponding relaxation parameter, computed from the closed-form solution, that minimizes the spectral radius of the iteration matrix. Consequently, we can transform the original joint optimization problem, with respect to both penalty and relaxation parameters, into a single-variable optimization problem focused only on the penalty parameter.

\item We verify our proposed parameter selection methods through random instantiations and practical real-world imaging applications, encompassing diffeomorphic image registration, image deblurring, MRI reconstruction. This approach sets us apart from previous methods, e.g., \cite{ghadimi2014optimal}, that only depend on simulated data for validation purpose.
\end{itemize}

\section{Related Works}
\cite{boley2013local} studied the convergence rate of ADMM for both quadratic and linear programs via the spectral analysis based on a novel matrix recurrence. While acknowledging that the penalty parameters of ADMM can influence its convergence rate, they did not offer guidance on how to select these parameters. To address this issue, \cite{ghadimi2014optimal} reformulated ADMM into a fixed-point iteration system to analyze the impact of parameters on the convergence rate of ADMM and over-relaxed ADMM. By minimizing the spectral radius of the iteration matrix, they successfully derived optimal penalty and relaxation parameters for quadratic programming. \cite{teixeira2015admm} extended the applicability of Ghadimi's theory by transforming the distributed quadratic programming into an equivalent constrained quadratic programming. \cite{francca2016explicit} introduced a method that determines the relaxation parameter for semi-definite programming through the analysis of the problem's condition number. 

\cite{boyd2011distributed} suggested an empirical parameter update strategy for ADMM's penalty parameters. The idea is to maintain a proportional relationship between the norms of primal and dual residuals, ensuring their convergence to zero within a specified factor. \cite{xu2017adaptive} proposed an adaptive ADMM approach by applying the Barzilai-Borwein spectral method to the original ADMM algorithm. Their method allows to dynamically update penalty parameters in each iteration based on primal and dual residuals. Inspired by this work, \cite{mavromatis2020auto} introduced a weighted penalty parameter ADMM algorithm for solving optimal power flow problems. Their approach involves the computation of absolute values from the admittance matrix and the Hessian matrix in each ADMM iteration. These values are then used to recalibrate the penalty parameters, aiming to refine the accuracy of parameter estimation.

However, certain limitations exist in the current research landscape. Firstly, many methods \cite{boyd2011distributed,xu2017adaptive,wohlberg2017admm,mhanna2018adaptive} rely on primal and dual residuals for estimating optimal parameters during iterations, but there often lack closed-form or explicit pre-iteration parameter selection approaches. Secondly, existing parameter selection techniques, based on the spectral analysis of the iteration matrix \cite{ghadimi2014optimal,francca2016explicit}, predominantly focus on specific problem types (e.g., standard quadratic problem with $L$ being an identity matrix). These methods requiring the spectral radius of the iteration matrix to be computable in an explicit form, which restricts their applicability and generalization ability \cite{stellato2020osqp}. In this paper, we will propose effective methods to address these two challenges.

\section{Methodology}
This section starts with the introduction of essential notations utilized in the subsequent formulations. We proceed by presenting the concept of fixed-point iterations, which serves as a foundational element for both the convergence analysis and parameter selection processes. Following this, we proceed to apply both ADMM and its over-relaxed variant to address LQPs. In the final stages, we propose novel methods for selecting the penalty and relaxation parameters. This is accomplished through the conversion of ADMM and over-relaxed ADMM into the form of fixed-point iterations, followed by the utilization of spectral radius analysis.

\subsection{Notations and Fixed-Point Iterations}
Let $\mathbb{R}$ and $\mathbb{C}$ denote respectively the set of real and complex numbers, $\mathbb{R}_{++}$ denote the set of positive numbers, $\mathcal{S}^{n\times n}$ denote the set of $n \times n$ matrices, and $I^n$ (or $I$) be the $n \times n$ identity matrix. For the square matrix $T$ and its corresponding eigenvalues $\lambda's$, we define the $n$th smallest eigenvalue of $T$ as $\lambda_n(T)$, and the spectral radius of $T$ as $\rho(T)$. 

Fixed-point iterations involve the iterative process below
\begin{equation*}
	u^{k+1}=Tu^{k}+c,
\end{equation*}
where $T \in \mathcal{S}^{n\times n}$ is known as the iteration matrix, $u\in \mathbb{R}^{n}$, and $c\in \mathbb{R}^{n}$. It was shown in  \cite{ghadimi2014optimal} that the convergence factor $\zeta$ of this fixed-point iteration system is equal to $\rho(T)$. Here, the convergence factor $\zeta$ is defined as
\begin{equation*}
	\zeta\triangleq \sup\limits_{k:u^k\neq u^*}\frac{\|u^{k+1}-u^*\|}{\|u^{k}-u^*\|},
\end{equation*}
where $\|\cdot\|$ represents the $L_2$ norm, and $u^{*}$ denotes the optimal solution (i.e., so-called ground truth).  The sequence $\{u_k\}$ is $Q$-sublinear if $\zeta=1$, $Q$-linear if $\zeta<1$, and $Q$-superlinear if $\zeta=0$. Throughout this paper, the letter $Q$ has been omitted when referring to the convergence rate. For linearly convergence sequences with $\zeta\in (0,1)$, if we define $t_{\epsilon}$ as the smallest iteration count to ensure $\|u^{k+1}-u^{*}\|<\varepsilon$ for all $k>t_{\varepsilon}$, then $t_{\varepsilon}$ can be calculated by $\left(\log(\varepsilon)-\log(\sigma)\right)/ \log(\zeta)$, where $\sigma$ denotes the worst case distance between $u^{0}$ and $u^{*}$, i.e., $\|u^{0}-u^{*}\|<\sigma$. This suggests that by reducing the value of the the convergence factor $\zeta$, the iteration count can be decreased, leading to a faster convergence rate. 

\subsection{ADMM for LQPs} 
The LQPs for image processing we study in this paper have the following structure
\begin{equation}\label{QP}
	\min\limits_{u}\frac{\mu}{2}\|Au-f\|^2+\frac{1}{2}\|L u\|^2,
\end{equation}
where $\mu\in\mathbb{R}_{++}$ is the regularization parameter; $A\in \mathbb{R}^{m\times n}$ or $\mathbb{C}^{m\times n} \left(m \leq n\right)$ is an encoding matrix; $u\in \mathbb{R}^{n}$ or  $\mathbb{C}^{n}$ is the unknown vector; $f\in \mathbb{R}^{m}$ or $\mathbb{C}^{m} $ is the input vector; and $L\in \mathbb{R}^{n\times n}$ is a regularization matrix. The value of $\mu$ determines the output quality, whereas smaller values of $\mu$ tend to yield smoother results. By differentiating \eqref{QP} with respect to $u$ and setting the respective derivative to zero, we have the following linear system
\begin{equation}\label{NormalEq}
	\left(\mu A^TA + L^TL\right) u = \mu A^T f.
\end{equation}
When addressing the solution of Equation \eqref{NormalEq}, two primary challenges arise: Firstly, in certain scenarios like our MRI reconstruction and diffeomorphic image registration, where $\left(\mu A^TA + L^TL\right)$ may be positive semi-definite, the process of inverting such a matrix becomes unfeasible. Secondly, in the context of higher-dimensional cases like 3D medical image registration \cite{thorley2021nesterov}, even if the matrix $\left(\mu A^TA + L^TL\right)$ remains positive definite, the process of matrix inversion becomes computationally expensive. To address these two issues, we propose to use ADMM to handle the original problem \eqref{QP}, as an alternative to using the normal equation to solve \eqref{NormalEq}.

To apply ADMM, we introduce an auxiliary variable $w \in \mathbb{R}^{n}$, a Lagrangian multiplier $b \in \mathbb{R}^{n}$, and a penalty parameter $\theta\in\mathbb{R}_{++}$, transforming \eqref{QP} into the following augmented Lagrange function
\begin{equation}\label{ADMM2}
{\cal L}({u,w; b}) = \frac{\mu}{2}\|Au-f\|^2+\frac{1}{2}\|L w\|^2+\frac{\theta}{2}\|w-u-b\|^2.
\end{equation}
To optimize \eqref{ADMM2} with ADMM, we need to decompose it into two sub-problems with respect to $u$ and $w$ and then update the Lagrangian multipliers $b$ until the process converges. The following Algorithm 1 outlines the optimization process using ADMM.
\begin{table}[t]
	\label{GADMM}
	\centering
	\resizebox{0.47\textwidth}{!}{
	\begin{tabular}{l}
		\hline
		\textbf{Algorithm 1:} ADMM for LQPs               \\
		\hline
		\textbf{Input:} matrices $A$ and $L$;   parameter $\mu$ and $\theta$\\
		\textbf{Initialize:} $u^0$ and $b^0$  \\
		\textbf{Repeat:}      \\
		\qquad$w^{k+1}=\text{arg}\min\limits_{w}\frac{1}{2}\|L w\|^2+\frac{\theta}{2}\|w-u^k-b^k\|^2$\\
		\qquad $u^{k+1}=\text{arg}\min\limits_{u}\frac{\mu}{2}\|Au-f\|^2+\frac{\theta}{2}\|w^{k+1}-u-b^k\|^2$       \\
		
		\qquad$b^{k+1}=b^{k}+u^{k+1} - w^{k+1}$\\
		\textbf{until} some stopping criterion is met\\
		\hline
	\end{tabular}}
\end{table}
In Algorithm 1, we have $w^{k+1}=( {L^T}L +\theta I)^{-1} (\theta u^k + \theta b^k)$ and $u^{k+1}=(\mu A^T A+\theta I)^{-1} (\theta w^{k+1} - \theta b^k + \mu A^T f)$. It is worth noting that while matrix inversion is applied to both variables $w^{k+1}$ and $u^{k+1}$, fast solvers exist in specific cases due to the distinctive structure of $A^TA$ and ${L^T}L$. For instance, in diffeomorphic image registration, $A^TA$ takes on a rank-1 form, allowing efficient inversion through the Morris-Sherman equation \cite{bartlett1951inverse,thorley2021nesterov}. Similarly, in MRI reconstruction and diffeomorphic image registration, $L^TL$ can be effectively diagonalized using the discrete Fourier transformation basis functions \cite{goldstein2009split,duan2023arbitrary}. Consequently, the application of ADMM to solve LQPs offers distinct advantages.

\begin{theorem}
In order to determine the optimal penalty parameter $\theta^*$ in ADMM automatically, we need to transform the ADMM iterations in Algorithm 1 into the following fixed-point iteration system, solely with respect to the variable $u$
\begin{equation}\label{u}
	u^{k+1} = \left(I+ Q\right) u^{k} - \left(\mu A^T A+\theta I\right)^{-1}\left(\theta \mu A^T f\right),
\end{equation}
where $I+Q$ is the iteration matrix with $Q$ defined as 
\begin{equation}\label{Q}
	Q = \theta (\mu A^T A+\theta I)^{-1} ( ( {L^T}L +\theta I)^{-1} (\theta I -  \mu A^T A ) - I).
\end{equation}
Next, given a value of $\mu$, we can prove 
\begin{equation}\label{zeta}
	\rho\left(I+Q\right)\leq1,
\end{equation}
regardless of the value of $\theta$. As per Section 3.1, we know that the convergence factor $\zeta$ of Algorithm 1 is equal to the spectral radius of the iteration matrix. As such, $\zeta$ is bounded by 1, meaning Algorithm 1 or \eqref{u} is unconditionally convergent.

\end{theorem}

\begin{proof}
 Detailed derivations proving the equivalence between Algorithm 1 and the fixed-point iteration system \eqref{u}, as well as Inequality \eqref{zeta}, have been provided in  Appendix 1 of the arXiv version of this paper.
\end{proof}

Next, we search the optimal parameter $\theta^*$ that minimizes the convergence rate of Algorithm 1. Since $\zeta$ is dependent on the penalty parameter $\theta$, the objective is to identify a value for $\theta$ that minimizes the convergence factor $\zeta$. For this, we define the following minimization problem
\begin{equation} \label{mintheta}
	\min_\theta \zeta\left(\theta\right),
\end{equation}
where $\zeta\left(\theta\right) = \rho \left(I+ Q(\theta)\right)$. From Inequality \eqref{zeta} in Theorem 1 we have $\lambda_i\left(Q(\theta)\right) \in [-1,0]$, and we can also easily derive $\lambda_i \left(I+ Q(\theta)\right) = 1 +  \lambda_i\left(Q(\theta)\right)$. As such, we have $\rho \left(I+ Q(\theta)\right) = 1 +  \lambda_n\left(Q(\theta)\right)$, with which the minimization problem \eqref{mintheta} can be converted to 
\begin{equation} \label{mintheta_1}
	\min_\theta \lambda_n\left(Q(\theta)\right).
\end{equation}

Though the minimization problem \eqref{mintheta_1} is a one-dimensional optimization problem with respect to only $\theta$, computing $\theta^*$ directly is however not trivial. This is the reason why a general applicable method for optimizing $\theta$ is still lacking. Previous works \cite{ghadimi2014optimal,teixeira2015admm} were based on the assumption that $\lambda_n(Q)$ can be explicitly written for spectral analysis. However, in practical applications such as diffeomorphic registration in Section 4.2, this is a significant limitation. To address this challenge, we propose to use numerical gradient descent to optimize $\theta$
\begin{equation}\label{General}
	\begin{split}
{\theta ^{k + 1}} =  {\theta ^k} - t\nabla \lambda_n\left(Q(\theta^k)\right) ,
 \end{split}
\end{equation}
 where $t$ denotes the step size. In this study, we employed the central finite difference scheme to compute gradients. Compared to the one-sided finite difference method, this scheme offers better numerical stability. It also provides more accurate estimation of gradients. It is important to note that this gradient descent method is general, as it does not need to know the explicit form of the eigenvalues of matrix $Q$. The definition for the central finite difference is given by
\begin{equation*}
	\nabla \lambda_n\left(Q(\theta^k)\right)  \approx \frac{{\lambda_n\left(Q(\theta^k+ \eta)\right) - \lambda_n\left(Q(\theta^k - \eta)\right)}}{2\eta },
\end{equation*}
where $\eta$ represents a small value. In our experiments, we set this value within the range of $10^{-5}$ to $10^{-3}$, which led to a satisfactory convergence of the gradient descent~\eqref{General}.

\subsection{Over-Relaxed ADMM}
Over-relaxation technique can be used in the ADMM algorithm and further accelerate the convergence rate of ADMM. This method is achieved by introducing an additional relaxation parameter $\alpha$ and replacing $w^{k+1}$ in Algorithm 1 with $\alpha w^{k+1}+\left(1-\alpha\right)u^k$. Algorithm 2 outlines  the optimization process of the augmented Lagrange function~\eqref{ADMM2} using over-relaxed ADMM.
\begin{table}[t]
	\label{OR}
	\centering
	\resizebox{0.47\textwidth}{!}{
	\begin{tabular}{l}
		\hline
		\textbf{Algorithm 2:} Over-relaxed ADMM for LQPs                  \\
		\hline
		\textbf{Input:} matrices $A$ and $L$; parameter $\mu$, $\theta$ and $\alpha$ \\
		\textbf{Initialize:} $u^0$ and $b^0$  \\
		\textbf{Repeat:}      \\
		\qquad$w^{k+1}=\text{arg}\min\limits_{w}\frac{1}{2}\|L w\|^2+\frac{\theta}{2}\|w-u^k-b^k\|^2$\\
		\qquad $u^{k+1}=\text{arg}\min\limits_{u}\frac{\mu}{2}\|Au-f\|^2+\frac{\theta}{2}\|\alpha w^{k+1}-\alpha u^k-b^k\|^2$       \\
		\qquad$b^{k+1}=b^{k}+u^{k+1} - \alpha w^{k+1} - \left(1-\alpha\right)u^k$\\
		\textbf{until} some stopping criterion is met\\
		\hline
	\end{tabular}}
\end{table}

To investigate the influence of relaxation parameter $\alpha$ on convergence, we transform Algorithm 2 into its fixed-point iteration system. Such a conversion approach is in line with Proof of Theorem 1 in Appendix. The resulting fixed-point iteration system is given as follows
\begin{equation*}
	u^{k+1} = \left(I+ \alpha Q\right) u^{k}- \alpha\left(\mu A^T A+\theta I\right)^{-1}\left(\theta\mu A^T f\right).
\end{equation*}
After obtaining the iteration matrix $I+ \alpha Q$, we can analyze the spectral radius of this matrix to determine the optimal relaxation parameter $\alpha^*$.

\begin{theorem}
The optimal $\alpha^*$ can be directly calculated using the following closed-form formula
\begin{equation}\label{alpha}
	\alpha^* = -\frac{2}{\lambda_1\left(Q\left(\theta\right)\right)+\lambda_n\left(Q\left(\theta\right)\right)},
\end{equation}
where $Q(\theta)$ is a matrix whose entries reply on the value of $\theta$. As per Equation \eqref{alpha}, we can compute the optimal relaxation parameter $\alpha^*$ as long as a value of $\theta$ is given. 
\end{theorem}

\begin{figure}[t!]
\centering
\includegraphics[width=0.85\linewidth]{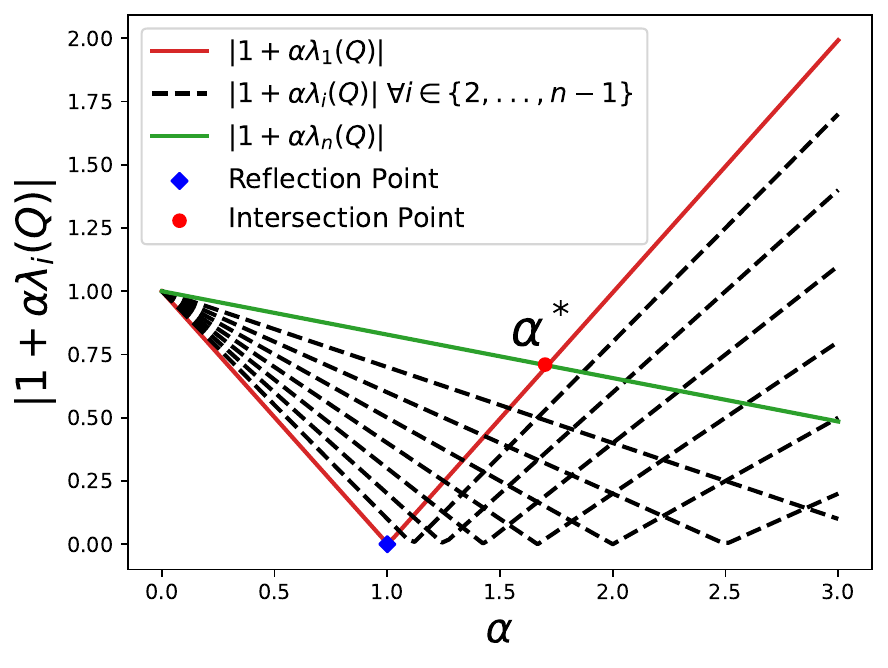}
\vspace{-5pt}
\caption{Relationship between $|1 + \alpha \lambda_i(Q)|$ and the value of $\alpha$. The slope of each line before reflection is $\lambda_i\left(Q\right)$. The spectral radius before the intersection point is governed by the green line, while after the reflection, it is determined by the reflected red line. The intersection point corresponds to the optimal $\alpha^*$ as well as the minimum spectral radius of the iteration matrix $I + \alpha Q
$.}
\vspace{-5pt}
\label{fig:1}
\end{figure}

\begin{proof} To prove Theorem 2, we begin with the following two-dimensional joint optimization problem
\begin{equation}\label{joint}
	\min_{\theta, \alpha} \zeta\left(\theta, \alpha \right),
\end{equation}
where $\zeta\left(\theta, \alpha \right)=\rho(I+\alpha Q(\theta))$. In order to express the spectral radius in terms of the eigenvalue structure, we first derive the equality $ \lambda_i \left(I+\alpha Q(\theta)\right) = 1 + \alpha \lambda_i\left(Q(\theta)\right)$, and the spectral radius $\rho(I+\alpha Q(\theta))$ is then defined as 
\begin{equation}\label{rho1}
  \max_i |1 + \alpha \lambda_i\left(Q(\theta)\right)|,\; \forall i \in \{1,...,n\}. 
\end{equation}

From Inequality \eqref{zeta} in Theorem 1, we know $\lambda_i\left(Q(\theta)\right) \in [-1,0]$. Based on this and \eqref{rho1}, we plot Figure~\ref{fig:1} to demonstrate the correlation between the absolute eigenvalue of the iteration matrix and the relaxation parameter. From this figure, it is straightforward to express the spectral radius as the following piecewise function
\begin{equation}\label{rho2}
	\rho= 
	\begin{cases}
		1+\alpha \lambda_n (Q), & \text{if}\;-1 - \alpha \lambda_1(Q)\leq 1 + \alpha \lambda_n(Q)\\
		-1-\alpha \lambda_1 (Q), & \text{if}\;-1 - \alpha \lambda_1(Q) > 1 + \alpha \lambda_n(Q)
	\end{cases},
\end{equation}
where with a slight abuse of notation, we use $\rho $ to represent $\rho(I+\alpha Q(\theta))$. 

Once we have \eqref{rho2}, our objective is to minimize it in order to enhance the convergence rate. From Figure~\ref{fig:1} again, it becomes evident that the smallest spectral radius is located at the intersection point where the following equality holds 
\begin{equation*}
    -1 - \alpha \lambda_1(Q) = 1 + \alpha \lambda_n(Q),
\end{equation*}
from which $\alpha$ can be computed using a closed-form solution as follows
\begin{equation*}
  \alpha = -\frac{2}{\lambda_1\left(Q\left(\theta\right)\right)+\lambda_n\left(Q\left(\theta\right)\right)},
\end{equation*}
which exactly verifies the validity of Equation \eqref{alpha}. 
\end{proof}

If now we plug the optimal $\alpha^*$ into \eqref{rho2}, we can convert the joint minimization problem \eqref{joint} into the following minimization problem
 \begin{equation}\label{theta3}
 	\min_{\theta, \alpha} \zeta\left(\theta, \alpha \right) \Rightarrow \min_{\theta} \frac{\lambda_1 (Q\left(\theta\right))-\lambda_n (Q\left(\theta\right))}{\lambda_1\left(Q\left(\theta\right)\right)+\lambda_n\left(Q\left(\theta\right)\right)},
 \end{equation}
which is a single-variable optimization problem with respect to only $\theta$. This problem can be minimized with a numerical gradient descent method similar to Equation~\eqref{General}. Once $\theta^*$ is found, $\alpha^*$ can be computed using the closed-form solution~\eqref{alpha}. It is worth noting that even if $\theta$ is not optimal, $\alpha$ computed via \eqref{alpha} can still accelerate convergence.

\section{Experiments}
In this section,  we will first test the generalization ability of our proposed parameter selection method through random instantiations. Following that, we will apply the proposed parameter selection methods to diffeomorphic image registration, image deblurring, and MRI reconstruction. We will compare our optimal ADMM algorithm and over-relaxed variant (oADMM) with gradient descent (GD), gradient descent with Nesterov's acceleration (GD-N) \cite{nesterov1983method,bartlett2021accelerated}, gradient descent with Nesterov's acceleration and restart (GD-NR) \cite{o2015adaptive,bartlett2021accelerated}, as well as conjugate gradient (CG). In all of our experiments, we chose the step size in gradient-based methods using the Lipschitz constant of the corresponding problem. It is worth noting that optimal values for penalty parameters can be determined analytically for image deblurring and MRI reconstruction problems. However, for image registration numerical gradient descent is required to compute these parameters.

\subsection{Generalization Ability}
Emphasizing that our approach is model-driven, the selection of parameters\footnote{The parameter selection also relies on the regularization parameter $\mu$ which we fix as a constant in this paper.} relies on the matrices $A$ and $L$ in the minimization problem \eqref{QP}. As such, the measure of generalization ability lies in how effectively our method performs as $A$ and $L$ undergo variations, which is in contrast to data-driven methods, where the generalization ability is often examined using multiple different datasets.

We presented Figure~\ref{fig:c} to demonstrate the generalization ability of our approach, where the analysis is based on 50 random instantiations of $A\in\mathbb{R}^{200\times 50}$ and $L\in\mathbb{R}^{200\times 50}$ while keeping $f$ and $\mu$ fixed. For ADMM, we employed numerical gradient descent to minimize \eqref{mintheta_1} with respect to $\theta$. For oADMM, we utilized numerical gradient descent to minimize \eqref{theta3} with regard to $\theta$, whilst the optimal value of $\alpha$ was calculated using \eqref{alpha} once $\theta^*$ was found. We note that the optimal values of $\theta$ for both ADMM and oADMM are similar and that the optimal values of $\alpha$ are not within $[1.5,1.8]$ as suggested in \cite{eckstein1994parallel}. As evident from Figure~\ref{fig:c}, the calculated optimal values consistently result in faster convergence rates for both ADMM and oADMM, reaffirming the generalization ability of our proposed parameter selection methods.

\begin{figure}[h!]
\centering
\includegraphics[width=1\linewidth]{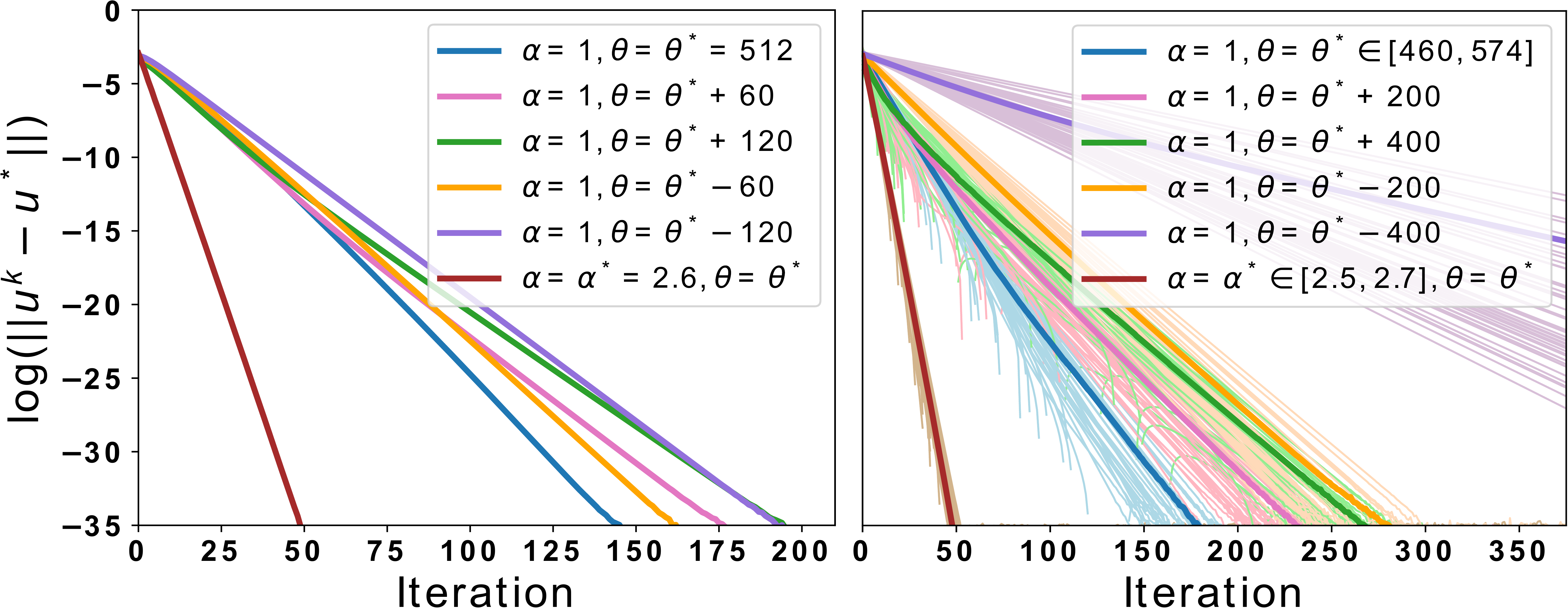}
\vspace{-17pt}
\caption{Left: Convergence rates of different methods and parameter values based on 1 random instantiation of $A$ and $L$. Right: Convergence rates based on 50 random instantiations of $A$ and $L$. The solid lines represent the average over 50 instantiations. The algorithm is ADMM when $\alpha=1$, and oADMM when $\alpha = \alpha^*$. }
\vspace{-5pt}
\label{fig:c}
\end{figure}

\subsection{Diffeomorphic Image Registration} 
Computing a diffeomorphic deformation can be treated as modelling a dynamical system \cite{beg2005computing}, given by an ordinary differential equation (ODE): $\partial{\bf{\phi}}/\partial t  = {\bf{v}}_t({\bf{\phi}}_t)$, where ${\bf{\phi}}_0 = {\rm{Id}}$ is the identity transformation and ${\bf{v}}_t$ indicates the velocity field at time $t$ ($\in [0,1]$).  The ODE can be solved by Euler integration, in which the deformation field $\phi$ is calculated as the compositions of a series of small deformations, defined as $\phi=({\rm{Id}}+\frac{{\bf{v}}_{t_{N-1}}}{N})
\circ \cdots \circ ({\rm{Id}}+\frac{{\bf{v}}_{t_{1}}}{N})\circ ({\rm{Id}}+\frac{{\bf{v}}_{t_0}}{N})$. If the velocity fields ${\bf{v}}_{t_{i}}$ are sufficiently small whilst satisfying some smoothness constraints, the resulting composition is a diffeomorphic deformation. 
    
To compute the velocity fields whilst satisfying these diffeomorphic constraints, we minimize the following linear quadratic problem \cite{thorley2021nesterov}
\begin{equation} \label{eq:OF}
\min_{v_x, v_y} \frac{\mu}{2} \| \langle  I_x, {v_x} \rangle + \langle I_y, {v_y} \rangle + I_t   \|^2 + \frac{1}{2} \|\nabla v_x\|^2 + \frac{1}{2}  \|\nabla v_y\|^2,
\end{equation}
where $I_x, I_y\in \mathbb{R}^{n}$ denote the spatial derivatives of the image; $I_t\in \mathbb{R}^{n}$ represents the temporal derivative of the image; and $v_x, v_y \in \mathbb{R}^{n}$ denote the velocity field in $x$ and $y$ directions. In this case, by setting 
$$A=\begin{pmatrix}
\textup{diag}(\langle I_x, I_x \rangle ) & \textup{diag}(\langle I_x, I_y \rangle )\\ 
\textup{diag}(\langle I_y, I_x \rangle ) & \textup{diag}(\langle I_y, I_y \rangle )
\end{pmatrix} \in \mathbb{R}^{2n\times2n} $$
and
$$L=\begin{pmatrix}
\nabla ^T\nabla  & 0\\ 
0 & \nabla ^T\nabla  
\end{pmatrix}\in \mathbb{R}^{2n\times2n},$$
we can use numerical gradient descent to compute optimal parameters for both ADMM and oADMM.

In Figure~\ref{fig:4}, we show results obtained through the introduced diffeomorphic registration technique. We examine the impact of the penalty parameter $\theta$ in both ADMM and oADMM, and then evaluate the convergence efficiency of different algorithms. Given a pair of images (depicted as source and target in the figure), we can compute a deformation (shown in the bottom left panel) that ensures a positive Jacobian determinant (shown in the bottom middle panel) for all pixel positions. In the top right panel, we show the correlation between the spectral radius of the iteration matrix and $\theta$ in both ADMM and oADMM. As can be seen, there exists an unique optimal value where the spectral radius is minimized. As such, when using numerical gradient descent, it is possible to find the optimal value of $\theta$ that can considerably reduce iteration counts. This panel also illustrates that $\theta^*$, producing the smallest spectral radius for oADMM, closely aligns with that of ADMM. Furthermore, due to the two-loop\footnote{We did not use the pyramid implementation as in \cite{thorley2021nesterov}, so we ended up with a two-loop algorithm comprising inner ADMM/oADMM iterations and outer warping iterations.} iterative nature of diffeomorphic image registration, the data term of $\eqref{eq:OF}$ undergoes slight changes at each iteration of the outer loop. These changes however do not significantly influence the value of $\theta^*$, as evident from the top right panel. Therefore, given a specific value of $\mu$, it is sufficient to use gradient descent to search $\theta^*$ for each outer iteration. Finally, in the bottom right panel, convergence rates among different algorithms are compared. As is evident, the parameter-optimized oADMM algorithm remains the fastest in terms of convergence rate.
\begin{figure}[t]
	\centering
   \includegraphics[width=0.98\linewidth]{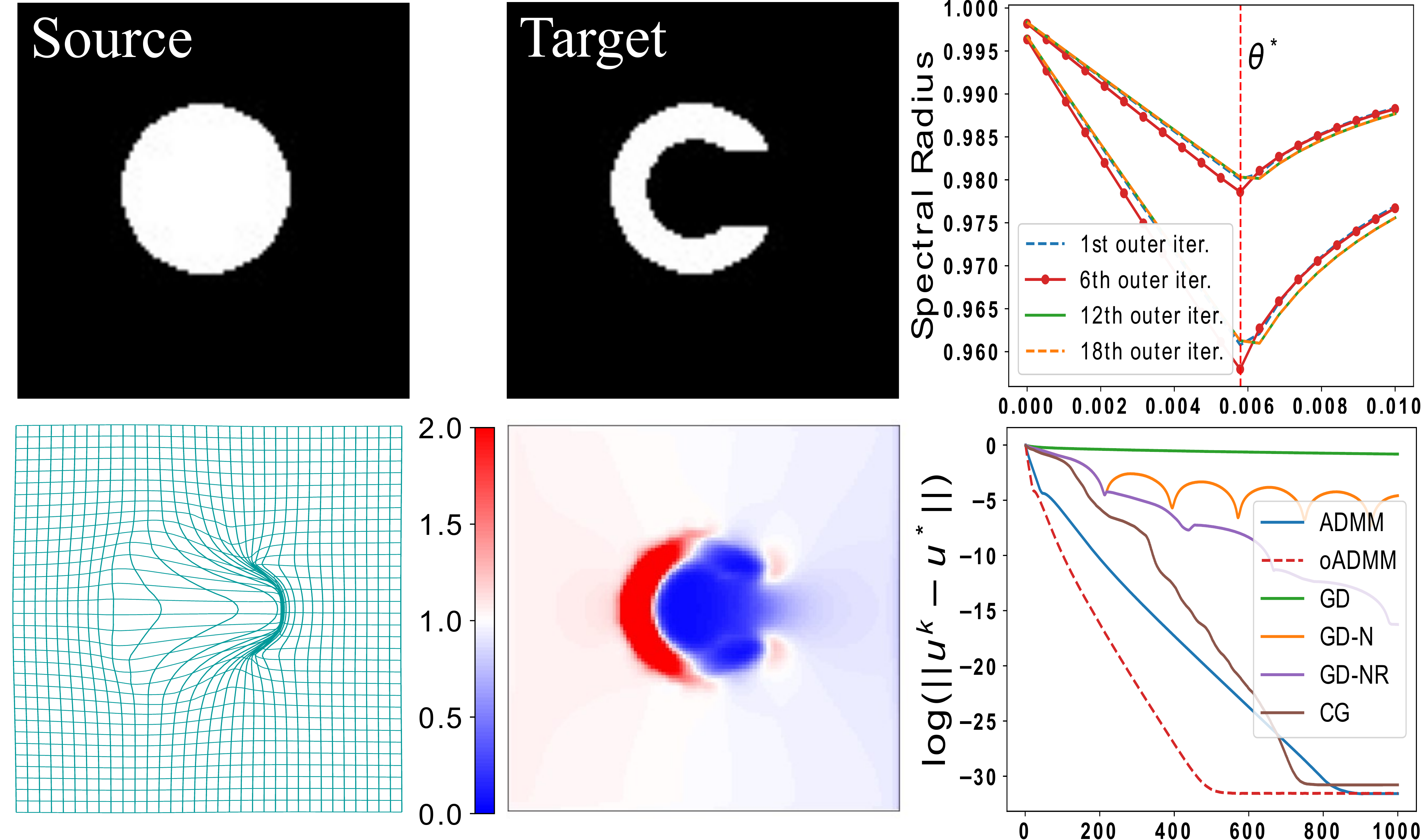}
	\vspace{-5pt}
	\caption{ Illustration of diffeomorphic image registration results, visualization of the correlation between spectral radius and $\theta$, and comparison of convergence rates of algorithms. The $x$-axes of the two plots in the third column represent the values of $\theta$ and iteration numbers, respectively.}
	\label{fig:4}
 \vspace{-10pt}
\end{figure}

\subsection{Image Deblurring}
In this application, we look at a phantom test image. The image went through a Gaussian blur of size $7\times7$ and standard deviation 2, followed by an additive zero-mean white Gaussian noise with standard deviation $10^{-4}$. The top left and middle panels of Figure~\ref{fig:2} depict the original and blurred images, respectively. To deblur the image we minimize the following problem
\begin{equation}\label{deblurring2}
	\min\limits_{u}\frac{\mu}{2}\|K u-f\|^2+\frac{1}{2}\|u\|^2,
\end{equation}
where $K \in \mathcal{S}^{n\times n}$ is the matrix representing the blur operator, $u\in \mathbb{R}^{n}$ is the vectorized unknown clean image, and $f\in \mathbb{R}^{n}$ is the vectorized input image.  By setting $A=K$ and $L=I$, the matrix $Q$ in \eqref{Q} for this application has the form of  
\begin{equation*}
	Q=\theta (\mu K^T K+\theta I)^{-1} (( I +\theta I)^{-1} (\theta I -  \mu K^T K )- I).
\end{equation*}

Since $K$ is a convolution matrix derived from the Gaussian kernel function, the eigenvalues of $K^T K$ can be calculated using the two-dimensional discrete Fourier transform \cite{capus2003fractional}. With $\lambda\left(K^T K\right)$, we can derive the maximum eigenvalues of $Q$ as
\begin{equation} \label{app1Q}
	\lambda_n(Q) =- \frac{\theta +\theta\mu \lambda_i(K^T K) }{\theta^2+\mu \lambda_i(K^T K) +\theta +\theta\mu \lambda_i(K^T K)},
\end{equation}
where $ i$ is either $1$ or $n$. Since $\lambda_n (Q)$ in this case can be explicitly written, we can derive closed-form solutions for the parameters in ADMM and over-relaxed ADMM. In Theorem 3, we give their optimal parameters.

\begin{theorem}
Firstly, to tackle the optimization problem \eqref{deblurring2} using ADMM, given a regularization parameter $\mu\in \mathbb{R}_{++}$, the optimal value of the penalty parameter in ADMM can be expressed in closed form
\begin{equation*}
	\theta^*=
	\begin{cases}
		\sqrt{\mu},\quad & {\rm{if}}\;\mu \leq 1 \\
		1,\quad & \text{otherwise}
	\end{cases},
\end{equation*}
which was derived by minimizing the value of $\lambda_n (Q)$ in \eqref{app1Q}. 

If over-relaxed ADMM is used to tackle the optimization problem \eqref{deblurring2}, the optimal penalty and relaxation parameters are given by
\begin{equation*}
	\theta^*=1 \; \rm{and} \; \alpha^*=2,
\end{equation*}
among which $\theta^*$ was determined by minimizing the problem \eqref{theta3} with $\lambda_1(Q)$ and $\lambda_n(Q)$ defined in \eqref{app1Q}, and then $\alpha^*$ was computed using \eqref{alpha} with $\theta^*$.
\end{theorem}

\begin{proof}
Detailed derivations have been given in  Appendix 2 of the arXiv version of this paper. 
\end{proof}

\begin{figure}[t]
	\centering
        \includegraphics[width=1.0\linewidth]{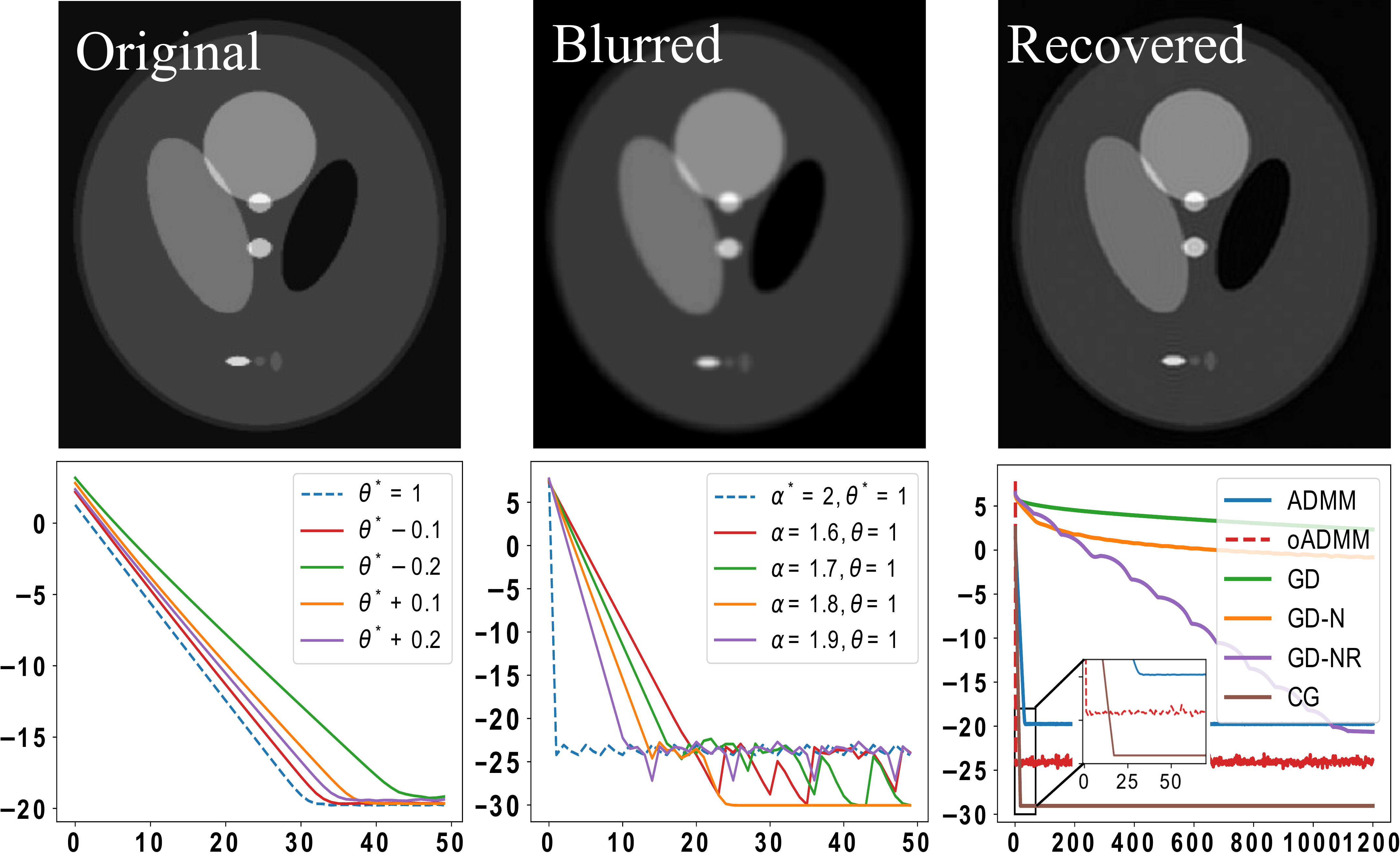}
        \vspace{-15pt}
	\caption{{Demonstration of image deblurring effects and convergence rates of different algorithms. The $x$-axis and $y$-axis of each plot in the second row represent iteration numbers and ${\rm{log}}(\|u^k - u^*\|)$, respectively. } }
	\label{fig:2}
 \vspace{-10pt}
\end{figure} 

In Figure~\ref{fig:2}, the top right panel displays the deblurred image from oADMM (comparable results were achieved with ADMM), which closely resembles the original image. Note that we set the regularization parameter $\mu$ to $10^3$ for this experiment. The bottom left and middle panels demonstrate that with the optimal $\theta^*$ and $\alpha^*$ there is a clear enhancement over the model's convergence, and that the optimal values of $\theta$ for both ADMM and oADMM are the same in this case. The bottom right panel shows that ADMM, CG, and oADMM exhibit superior performance than GD, GD-N, and GD-NR. Upon a detailed examination from the zoomed-in window, CG addresses this quadratic problem very well, albeit still needing multiple iterations to attain convergence. In contrast, oADMM achieves convergence in a single step\footnote{The spectral radius in this case is close to zero, leading to a superlinear convergence rate.}, outperforming all compared algorithms. 

\subsection{MRI Reconstruction}  
To reconstruct MR images we minimize the problem
\begin{equation}\label{MRIt2}
	\min\limits_{u}\frac{\mu}{2}\|\mathcal{D} \mathcal{F}u-f\|^2+\frac{1}{2}\|\nabla u\|^2,
\end{equation}
where $\mathcal{D}\in\mathbb{R}^{m\times n} \left(m<n\right)$ is the sampling matrix; $\mathcal{F}\in\mathbb{C}^{n\times n}$ is the Fourier transform matrix; $u\in\mathbb{C}^{n}$ is a complex-valued MR image stacked as a column vector; $f\in\mathbb{C}^{m}$  is the undersampled $k$-space data; $\nabla$ denotes the first-order gradient operator.
By setting $A=\mathcal{D} \mathcal{F}$ and $L=\nabla$, the matrix $Q$ in \eqref{Q} for this application has the form of
\begin{equation}\label{iterMRI}
	 Q=\theta(\mu M_1 +\theta I)^{-1} ( ( {\nabla^T}\nabla +\theta I)^{-1} (\theta I - \mu M_1  ) - I),
\end{equation}
where $M_1=\mathcal{F}^T \mathcal{D}^T \mathcal{D} \mathcal{F}$. Due to the use of periodic boundary conditions, $\nabla^T\nabla$ can be efficiently diagonalized in the form of $\mathcal{F}^T G \mathcal{F}$, where $G$ is a diagonal matrix. Equation \eqref{iterMRI} can be simplified to $\mathcal{F}^T M_2 \mathcal{F}$, where $M_2$ is given as
\begin{equation*}\label{iterdeblurring}
	 M_2 = \theta(\mu \mathcal{D}^T \mathcal{D}+\theta I)^{-1} ( ( G +\theta I)^{-1} (\theta I -  \mu \mathcal{D}^T \mathcal{D} ) - I),
\end{equation*}

\noindent which is a diagonal matrix. The eigenvalues of $\mathcal{F}^T M_2 \mathcal{F}$ are simply the values along the diagonal of $M_2$. If we define $\lambda_i(G)$ as the $i$th smallest eigenvalue of $G$, and $d_i$ as the diagonal value of $\mathcal{D}^T \mathcal{D}$ at the position where $\lambda_i(G)$ is indexed from $G$, the maximum eigenvalues of $Q$ can be derived as 
\begin{equation*}
	\lambda_n(Q) = -\frac{\theta \lambda_i(G)+\theta\mu d_i}{\theta^2+\mu d_i \lambda_i(G)+\theta \lambda_i(G)+\theta\mu d_i},
\end{equation*}

where $i \in \{1,...,n\}$. Since $\lambda_n (Q)$ can be written explicitly, we can derive the closed-form solution for $\theta$ in ADMM. In Theorem 3, we present the optimal value for this parameter. If over-relaxed ADMM is used to solve \eqref{MRIt2}, a closed-form solution still exists for $\theta$. It is however too cumbersome to derive them in this case. As such, the penalty parameter $\theta$ in over-relaxed ADMM was searched by gradient descent, and once $\theta^*$ was found the optimal relaxation parameter $\alpha^*$ can be directly obtained using Equation \eqref{alpha} with $\theta^*$.

\begin{theorem}
To tackle the optimization problem \eqref{MRIt2} using ADMM, given a regularization parameter $\mu\in \mathbb{R}_{++}$, the optimal value of the penalty parameter in ADMM can be expressed in closed form
\begin{equation*}
	\theta^*=
	\begin{cases}
		\sqrt{\mu a},\quad & \text{if}\quad  \mu \leq 2b- a\\
		\sqrt{\frac{\mu a b}{\mu+a-b}},\quad & \text{if}\quad  2b - a <\mu \leq a \\
		\mu,\quad & \text{if}\quad a < \mu \leq b\\
		\sqrt{\frac{\mu c b}{\mu+c-b}},\quad & \text{otherwise} 	
	\end{cases},
\end{equation*}
where $a$, $b$ and $c$ are defined as follows
\begin{equation*}
    \begin{split}
    a &= \lambda_1((\mathcal{D}^T \mathcal{D}) \odot G)\\
    b &= \lambda_1( (1 -\mathcal{D}^T \mathcal{D}) \odot G),\\
    c &= \lambda_x( (\mathcal{D}^T \mathcal{D} )\odot G)\\ 
\end{split}
\end{equation*}
where $\odot$ is the hadamard product; $\lambda_1((\mathcal{D}^T \mathcal{D}) \odot G)$ denotes the smallest eigenvalue of $(\mathcal{D}^T \mathcal{D}) \odot G$, excluding the eigenvalues corresponding to zero entries on the diagonal of $\mathcal{D}^T \mathcal{D}$; 
$\lambda_1( (1 -\mathcal{D}^T \mathcal{D}) \odot G)$ denotes the smallest eigenvalue of $(1- \mathcal{D}^T \mathcal{D}) \odot G$, excluding the eigenvalues corresponding to zero entries on the diagonal $1-\mathcal{D}^T \mathcal{D}$; and
$\lambda_x((\mathcal{D}^T \mathcal{D}) \odot G)$ represents the largest eigenvalue of $(\mathcal{D}^T \mathcal{D}) \odot G$, excluding the eigenvalues corresponding to zero entries along the diagonal of $\mathcal{D}^T \mathcal{D}$.
\end{theorem}

\begin{proof}
Detailed derivations have been given in Appendix 3 of the arXiv version of this paper. 
\end{proof}

In Figure~\ref{fig:MRI}, we reconstruct a cardiac MR image from $k$-space. The original image (displayed in the top left panel) was first transformed into $k$-space using the Fourier transformation. Then $50\%$ of the data there was taken using a cartesian sampling mask, displayed in the original image. This undersampled data was then corrupted by an additive zero-mean white Gaussian noise with standard deviation $1$ to form $f$ in \eqref{MRIt2}. The reconstruction (top right), despite some slight blurring due to the smooth regularization, clearly enhances image quality compared to that displayed in the top middle panel, which is a direct reconstruction of $f$ using the inverse Fourier transformation. The bottom left and middle panels of this figure illustrate that the choice of $\theta$ and $\alpha$ has a significant impact on the convergence rate and that our proposed methods result in faster convergence. Thanks to the utilization of these optimal parameters, we observe a clear superiority of our ADMM and oADMM over GD, its accelerated variants, and CG in terms of convergence efficiency.

\begin{figure}[t]
\centering
\includegraphics[width=0.97\linewidth]{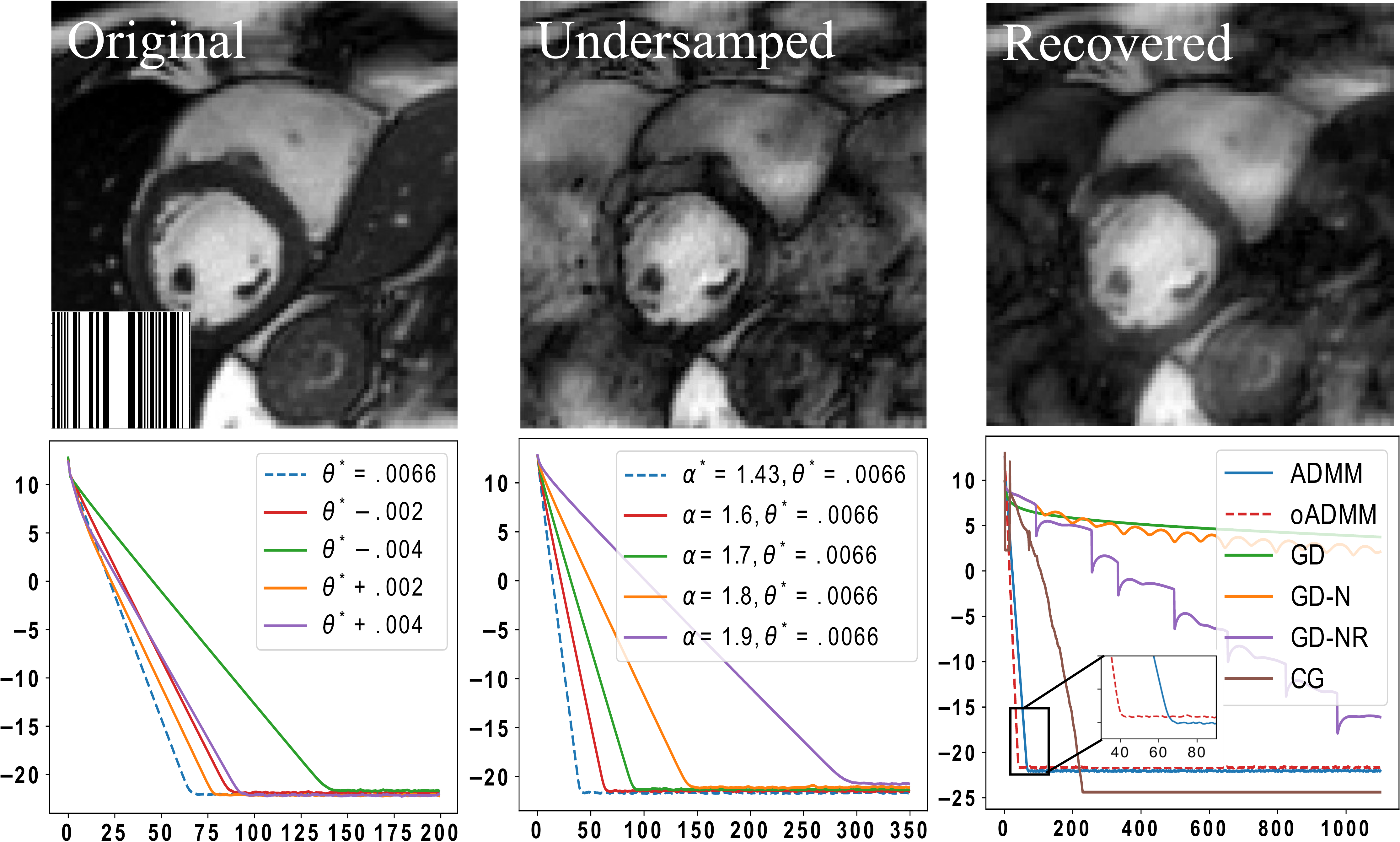}
\vspace{-7pt}
\caption{{Demonstration of MRI reconstruction results and comparison of convergence rates among algorithms. The $x$-axis and $y$-axis of each plot in the second row represent iteration numbers and ${\rm{log}}(\|u^k - u^*\|)$, respectively.} }
\label{fig:MRI}
\vspace{-13pt}
\end{figure}

\section{Conclusion}
In this paper, we presented automated techniques for selecting optimal penalty and relaxation parameters within the framework of ADMM and over-relaxed ADMM for linear quadratic problems. Our approaches involve a numerical gradient descent method for estimating the penalty parameter and a novel closed-form solution for determining the optimal relaxation parameter. We verified the generalizability and efficacy of these approaches through random instantiations and real-world imaging applications.

\subsection*{Appendix 1}
In Algorithm 1 of ADMM, the $u$-update, $w$-update, and $b$-update are respectively provided as follows
\begin{equation} \label{eq:ADMMw}
{w^{k + 1}} = {\left( { {L^T}L + \theta I} \right)}^{ - 1}\left( {\theta {u^k} + \theta {b^k}} \right),
\end{equation}
\begin{equation} \label{eq:ADMMu}
{u^{k + 1}} = {\left( {\mu{A^T}A + \theta I} \right)}^{ - 1}\left( {\theta {w^{k + 1}} - \theta {b^k} + \mu{A^T}f} \right),
\end{equation}
\begin{equation} \label{eq:ADMMb}
	{b^{k + 1}} = {b^k} + {u^{k + 1}} - {w^{k + 1}}.
\end{equation}
From Equation \eqref{eq:ADMMu} we obtain
\begin{equation*}
	{b^k}={w^{k + 1}}-u^{k + 1}- \frac{{\mu{A^T}A  }  u^{k + 1}}{\theta}+ \frac{{\mu{A^T}f}}{\theta },
\end{equation*}
with which Equation \eqref{eq:ADMMb} can be re-written as follows
\begin{equation*} 
	{b^{k + 1}} =  - \frac{{\mu{A^T}A}}{\theta }{u^{k + 1}} + \frac{{\mu{A^T}f}}{\theta },
\end{equation*}
which implies 
\begin{equation} \label{eq:ADMMb3}
	{b^{k}} =  - \frac{{\mu{A^T}A}}{\theta }{u^{k}} + \frac{{\mu{A^T}f}}{\theta }.
\end{equation}

By replacing $b^k$ in \eqref{eq:ADMMw} with \eqref{eq:ADMMb3}, $w$-update in \eqref{eq:ADMMw} can be re-written as
\begin{equation} \label{eq:ADMMw1}
w^{k + 1} = {( { {L^T}L + \theta I})^{ - 1}}( {( {\theta I - \mu{A^T}A}){u^k} - \mu{A^T}f}).
\end{equation}
Substituting $\eqref{eq:ADMMw1}$, along with \eqref{eq:ADMMb3}, into $\eqref{eq:ADMMu}$ yields the fixed iteration system below, solely with respect to the variable $u$
\begin{equation}\label{fixitu}
	u^{k+1} = \left(I+ Q\right) u^{k} - \left(\mu A^T A+\theta I\right)^{-1}\left(\theta \mu A^T f\right),
\end{equation}
where $I+Q$ is the iteration matrix\footnote{The iteration matrix may not be a symmetric matrix, so in theory its eigenvalues could be complex. However, in our applications we did not observe such an effect.} with $Q$ defined as
\begin{equation*}
    Q = \theta (\mu A^T A+\theta I)^{-1} ( ( {L^T}L +\theta I)^{-1} (\theta I -  \mu A^T A ) - I),
\end{equation*}
or equivalently
\begin{equation}\label{iterQ}
    Q = - (\mu A^T A + \theta I)^{-1} (L^T L + \theta I)^{-1} ( \theta \mu A^T A + \theta L^T L).
\end{equation}
The iteration matrix $I+Q$ can be derived as
\begin{equation}\label{iterIQ}
    I+Q = (\mu A^T A + \theta I)^{-1} (L^T L + \theta I)^{-1} ( \theta^2 I + \mu L^T L A^T A).
\end{equation}

So far, we have successfully transformed the ADMM iterations in Algorithm 1 into the fixed-point iterations \eqref{fixitu}. Next, we need to analyze the spectral radius of the iteration matrix to establish the convergence of ADMM. Recall that for the matrix $Q$ and its corresponding eigenvalues $\lambda's$, we define $\lambda_i(Q)$ as the $i$th smallest eigenvalue of $Q$, $\rho(Q)$ as the spectral radius of $Q$, and $\| \cdot \|$ as the spectral norm.

First, by adding \eqref{iterQ} and \eqref{iterIQ}, we obtain
\begin{equation*}
\begin{split}
    I+2Q =  (\mu A^T A + \theta I)^{-1} (L^T L + \theta I)^{-1} \\
   ( \theta^2 I + \mu L^T L A^T A-\theta L^T L - \theta \mu  A^T A),
\end{split}
\end{equation*}
which is equivalent to
\begin{equation*}
\begin{split}
   I+2Q & = (\mu A^T A + \theta I)^{-1} (L^T L + \theta I)^{-1} \\
   & ( \theta I-L^T L)( \theta I-\mu A^T A).
\end{split}
\end{equation*}

In theory, altering the sequence of matrix multiplication will yield a matrix that has the same eigenvalues to the original one. In this case, we define the matrix $S$ below
\begin{equation*}
    S =  (L^T L + \theta I)^{-1}( \theta I-L^T L)
   ( \theta I-\mu A^T A)(\mu A^T A + \theta I)^{-1},
\end{equation*}
which has the same eigenvalues as $I+2Q$. Now, if we analyze the spectral norm of $S$, we can have the following two inequalities
\begin{equation*}
\begin{split}
    \|S_1\| & =\| (L^T L + \theta I)^{-1}( \theta I-L^T L)\| \\
            & =\max_i\left|\frac{\theta-\lambda_i (L^T L)}{\theta+\lambda_i (L^T L)}\right| \leq 1, \forall i \in \{1,...,n\}.
\end{split}
\end{equation*}
\begin{equation*}
\begin{split}
		\|S_2\| & = \| ( \theta I-\mu A^T A)(\mu A^T A + \theta I)^{-1} \| \\
                    & = \max_i\left|\frac{\theta-\lambda_i (\mu A^T A)}{\theta+\lambda_i (\mu A^T A)}\right|\leq 1, \forall i \in \{1,...,n\}.
  \end{split}
\end{equation*}
Combining such two inequalities, we can draw the conclusion that
\begin{equation*}
		\rho(I + 2Q)=\rho(S)\leq \|S\|\leq\|S_1\| \|S_2\|=1.
\end{equation*}
As such, the following inequality holds
\begin{equation*}
		-1\leq 1+2\lambda_i(Q)\leq 1, \;\forall i \in \{1,...,n\},
\end{equation*}
which means 
\begin{equation*}
		\lambda_i (Q)\in [-1,0], \; \forall i \in \{1,...,n\},
\end{equation*}
and 
\begin{equation*}
    \rho(I+Q)\leq 1.
\end{equation*}
This essentially proves that the spectral radius of the iteration matrix is bounded by 1, thereby allowing us to conclude that the ADMM or Algorithm 1 for linear quadratic problems converges unconditionally.

 \subsection*{Appendix 2}
Given a value of the regularization parameter $\mu$, the maximum eigenvalue of $Q$ is of the following form
\begin{equation*}
	\lambda_n(Q(\theta)) =- \frac{\theta +\theta\mu \lambda_i(K^T K) }{\theta^2+\mu \lambda_i(K^T K) +\theta +\theta\mu \lambda_i(K^T K)}.
\end{equation*}

First, we derive $\theta^*$ in ADMM. It is evident that $\lambda_n(Q)$ increases monotonically as $i$ in $\lambda_i(K^T K)$ increases when $\theta\geq 1$, but decreases monotonically as $i$ in $\lambda_i(K^T K)$ increases when $\theta< 1$. Hence, we only need to discuss these two scenarios separately. Due to the fact that $\lambda_{1} (K^T K)$ close to 0 and $\lambda_n (K^T K)=1$, $\lambda_n (Q)$ can be re-written in the piecewise function as follows
  \begin{equation*}
  	\lambda_n(Q(\theta))=
 	\begin{cases}
  		-\dfrac{\theta+\theta\mu}{\theta^2+\mu+\theta +\mu\theta}\quad & \text{if}\;\theta < 1 \\
  		\qquad-\dfrac{1 }{\theta+1}\quad & \text{if}\;\theta\geq 1 
  	\end{cases}.
  \end{equation*}
Now we need to minimize $\lambda_n(Q(\theta))$ with respect to $\theta$. When $\theta\geq 1$, $\lambda_n(Q(\theta))$ increases monotonically when $\theta$ increases. The optimal penalty parameter in this case that minimizes $\lambda_n(Q)$ is given by
\begin{equation*}
\theta^*=1,\; \lambda_n (Q(\theta^*))\approx -\frac{1}{2}.
\end{equation*}
When $\theta<1$, it is straightforward to know
\begin{equation*}\label{Q2}
	\theta^*=\sqrt{\mu},\; \lambda_n \left(Q(\theta^*)\right)=-\frac{1+\mu }{2 \sqrt{\mu}+1 +\mu }.
\end{equation*}
So far, we have computed the optimal values of $\theta^*$ for the two cases. The next step is to determine which of the two cases has a smaller function value (e.g., smaller spectral radius). As such, we have 
\begin{equation*}
	\theta^*=
	\begin{cases}
		\sqrt{\mu},\quad & \text{if}\;-\frac{1+\mu }{2 \sqrt{\mu}+1 +\mu } \leq -\frac{1}{2} \\
		1,\quad & \text{otherwise}
	\end{cases},
\end{equation*}
which is equivalent to the following form
\begin{equation*}
	\theta^*=
	\begin{cases}
		\sqrt{\mu},\quad & \text{if}\;\mu \leq 1 \\
		1,\quad & \text{otherwise}
	\end{cases}.
\end{equation*}

Next, we optimize the parameters in over-relaxed ADMM. When $\theta\geq 1$, $\lambda_1(Q(\theta))$ and $\lambda_n(Q(\theta))$ are respectively defined as
\begin{equation*}
	\lambda_1(Q(\theta)) =- \frac{\theta +\theta\mu  }{\theta^2+\mu +\theta +\theta\mu },\; \lambda_n(Q(\theta))=-\frac{1}{\theta+1}.
\end{equation*}
In this case, the minimization problem (14) in the main paper can be converted to
\begin{equation*}
	\min_{\theta} \frac{\lambda_1 (Q\left(\theta\right))-\lambda_n (Q\left(\theta\right))}{\lambda_1\left(Q\left(\theta\right)\right)+\lambda_n\left(Q\left(\theta\right)\right)} \Rightarrow  \min_{\theta} \frac{(\theta -1)\mu}{2\theta+\mu+\theta\mu}.
\end{equation*}
When $\theta\geq 1$, the objective function value of the converted problem is greater or equal to zero. In such a case, the minimum of the objective function is achieved when $\theta^* = 1$. The optimal value of $\alpha^*$ is computed by plugging $\theta^* = 1$ into Equation (10) in the main paper, resulting in 
\begin{equation*}
	\alpha^* = -\frac{2}{\lambda_1\left(Q\left(\theta^*\right)\right)+\lambda_n\left(Q\left(\theta^*\right)\right)}=2.
\end{equation*}

Similarly, when $\theta\leq1$ , $\lambda_1(Q)$ and $\lambda_n(Q)$ are respectively defined as follows
\begin{equation*}
	\lambda_1(Q(\theta)) =-\frac{1}{\theta+1},\; \lambda_n(Q(\theta))=- \frac{\theta +\theta\mu  }{\theta^2+\mu +\theta +\theta\mu }.
\end{equation*}
In this case, the minimization problem (14) in the main paper can be converted to
\begin{equation*}
	\min_{\theta} \frac{\lambda_1 (Q\left(\theta\right))-\lambda_n (Q\left(\theta\right))}{\lambda_1\left(Q\left(\theta\right)\right)+\lambda_n\left(Q\left(\theta\right)\right)} \Rightarrow \min_{\theta} \frac{(1-\theta )\mu}{2\theta+\mu+\theta\mu}.
\end{equation*}
When $\theta \leq 1$, the objective function value of the converted problem is also greater or equal to zero. This implies that  $\theta^*= 1$ and the corresponding optimal over-relaxation parameters $\alpha^*=2$. As such, for both cases we have the optimal penalty and over-relaxation parameters given as follows
\begin{equation*}
	\theta^*=1 \; \rm{and} \; \alpha^*=2,
\end{equation*}
which verifies the validity of the closed-form solutions in Theorem 3.

 \subsection*{Appendix 3}
 We start from the following equation
 \begin{equation} \label{m2}
     \lambda_n(Q)  = \lambda_n({\cal F^T} M_2 {\cal F}) = \lambda_n(M_2),
 \end{equation}
where $M_2$ is a diagonal matrix defined as
 \begin{equation*}\label{iterdeblurring}
	 M_2=\theta(\mu \mathcal{D}^T \mathcal{D}+\theta I)^{-1} ( ( G +\theta I)^{-1} (\theta I -  \mu \mathcal{D}^T \mathcal{D} ) - I).
\end{equation*}
If we define a vector $v \in \mathbb{R}^n$, each element of which corresponds to a diagonal entry of $M_2$, we have
\begin{equation*}\label{iterdeblurring}
\begin{split}
	 v_i&=\theta \frac{1}{\mu(\mathcal{D}^T \mathcal{D})_{ii}+\theta}  \left(\frac{\theta -  \mu (\mathcal{D}^T \mathcal{D})_{ii} }{ G_{ii} +\theta }   - 1\right)\\
         &=-\frac{\theta G_{ii}+\theta\mu (\mathcal{D}^T \mathcal{D})_{ii}}{\theta^2+\mu (\mathcal{D}^T \mathcal{D})_{ii} G_{ii}+\theta G_{ii}+\theta\mu (\mathcal{D}^T \mathcal{D})_{ii}}\\
         &=  \begin{cases} 
            -\frac{\theta G_{ii}}{\theta^2+\theta G_{ii}}& \text{if}\;(\mathcal{D}^T \mathcal{D})_{ii}= 0\\
		-\frac{\theta G_{ii}+\theta\mu }{\theta^2+\mu G_{ii}+\theta G_{ii}+\theta\mu } & \text{if}\;(\mathcal{D}^T \mathcal{D})_{ii}= 1
	\end{cases}.
 \end{split}
\end{equation*}
According to Equation \eqref{m2}, the spectral radius of $Q$ is equivalent to the maximum value found on the diagonal of $M_2$, so we have
\begin{multline} \label{maxv}
    \lambda_n(Q)   = \lambda_n(M_2) = \max_i v_i   = \\ 
  \max_i \left \{ -\frac{\theta G_{ii}}{\theta^2+\theta G_{ii}},  -\frac{\theta G_{ii}+\theta\mu }{\theta^2+\mu G_{ii}+\theta G_{ii}+\theta\mu } \right\},
\end{multline}
subject to the constraints with respect to the binary matrix $\mathcal{D}^T \mathcal{D}$. We need to divide the maximization problem \eqref{maxv} into two separate maximization problems, as follows
\begin{equation} \label{max1}
\max_i \left \{ -\frac{ G_{ii}}{\theta+G_{ii}}\right\}  \; s.t. \; (\mathcal{D}^T\mathcal{D})_{ii} = 0, 
\end{equation}
and
\begin{equation} \label{max2}
\max_i \left \{ -\frac{\theta G_{ii}+\theta\mu }{\theta^2+\mu G_{ii}+\theta G_{ii}+\theta\mu }  \right\}  \; s.t. \; (\mathcal{D}^T\mathcal{D})_{ii} = 1.
\end{equation}

If we now introduce four new variables
\begin{equation*}
a = \arg \min_i G_{ii} \; s.t. \; (\mathcal{D}^T\mathcal{D})_{ii} = 1,
\end{equation*}
\begin{equation*}
b = \arg \min_i G_{ii} \; s.t. \; (\mathcal{D}^T\mathcal{D})_{ii} = 0,
\end{equation*}
\begin{equation*}
c = \arg \max_i G_{ii} \; s.t. \; (\mathcal{D}^T\mathcal{D})_{ii} = 1,
\end{equation*}
and 
\begin{equation*}
d = \arg \max_i G_{ii} \; s.t. \; (\mathcal{D}^T\mathcal{D})_{ii} = 0,
\end{equation*}
then \eqref{max1} and \eqref{max2} can be converted into the following two equivalent maximization problems
\begin{equation} \label{max1_}
    \max \left \{ -\frac{ b}{\theta+ b}, - \frac{d}{\theta + d}\right\} ,
\end{equation}
and
\begin{equation} \label{max2_}
    \max \left \{ -\frac{\theta a +\theta\mu }{\theta^2+\mu a+\theta a +\theta\mu }, -\frac{\theta c +\theta\mu }{\theta^2+\mu c+\theta c +\theta\mu }  \right\}.
\end{equation}
If we denote the solution for the maximization \eqref{max1_} as $x$, then we have
\begin{equation*} \label{x}
    x =  \max \left \{ -\frac{ b}{\theta+ b}, - \frac{d}{\theta + d}\right\}  =  -\frac{ b}{\theta+ b}. 
\end{equation*}
If we define the solution for the maximization \eqref{max2_} as $y$, then we have
\begin{equation*} \label{y}
    y = \begin{cases}
  		-\dfrac{\theta a+\mu\theta}{\theta^2+\mu a+\theta a+\mu\theta}\quad & \text{if}\;\theta> \mu    \\
  		-\dfrac{\theta c+\mu\theta}{\theta^2+\mu c+\theta c+\mu\theta}\quad & \text{if}\;\theta<\mu \\
            \qquad \quad -1/2\quad & \text{if}\;\theta=\mu \\
  	\end{cases}.
\end{equation*}
Next, we must compare $x$ and $y$ to determine identify the larger one as $\lambda_n(Q)$, and subsequently minimize it with respect to $\theta$. This results in the following discussions. 

First, we consider the case when $\;\theta>\mu\; \text{and}\;  \mu<a$. In this case, we have 
\begin{equation*}
F \triangleq x-y =\frac{\left(a + \mu - b\right)\theta^2 - \mu ab}{(\theta^2+\mu a+\theta a+\mu\theta)(\theta+b)}.
\end{equation*}
The denominator of $F$ is always greater than zero, so we only need to consider the sign of the numerator. In the case when $\mu\leq b-a$, $F <0 $, then we have $y > x$, leading to $\theta^*=\sqrt{\mu a}$. In the case when $\mu> b-a$, $F = 0$ holds only when $\theta$ satisfies the following equation
\begin{equation*}
\theta = \sqrt{\frac{\mu a b}{\mu+a-b}}.
\end{equation*}
If $\theta\geq e_1$ where $e_1$ denotes $\sqrt{\mu a b/(\mu+a-b)}$, we have $x\geq y$, then $\theta^*=e_1$. If $\theta < e_1$, then $x < y$. In this case, we should compute $\theta^*$ from $y$. If $e_1\geq \sqrt{\mu a}$, which is equivalent to $\mu\leq 2b-a$, then $\theta^*=\sqrt{\mu a}$. If $e_1 < \sqrt{\mu a}$, $\theta^* = e_1$. Consequently, we can conclude 
  \begin{equation}\label{theta1}
	\theta^*=
	\begin{cases}
            \sqrt{\mu a},\quad & \text{if}\;\mu\leq 2b-a  \\
		e_1,\quad & \text{if}\;2b-a<\mu\leq a 
	\end{cases}.
\end{equation}

Next, we consider the case when $\theta<\mu\; \text{and} \;\mu>c  $. In this case, we have 
\begin{equation*}
  	F = x-y =\frac{\left(c + \mu - b\right)\theta^2 - \mu cb}{(\theta^2+\mu c+\theta c+\mu\theta)(\theta+b)},
  \end{equation*}
where $c \geqslant b $. $F$ in this case increases monotonically in $\theta$. Moreover, $F=0$ holds only when $\theta$ satisfies the following equation
 \begin{equation*}
 	\theta = \sqrt{\frac{\mu c b}{\mu+c-b}}.
 \end{equation*}
As $c \geqslant b $, it is easy to know $e_2 < \sqrt{\mu c}$ in which $e_2$ denotes $\sqrt{\mu cb/(\mu+c-b)}$. In this case, when $\theta \leq e_2$, $y$ monotonically decreases in $\theta$. As such, $y$ achieves its minimum when $\theta^*=e_2$. This implies that when $\mu> c$, the optimal value $x$ and $y$ are exactly the same, which reads
\begin{equation}\label{theta2}
\theta^* =e_2,\;\text{if} \;\mu>c.
\end{equation}

Lastly, we consider the case when $\theta=\mu$ and $a\leq \mu \leq c$. In this case, we have 
\begin{equation*}
F \triangleq x-y =\frac{\theta-b}{2(\theta+b)}.
\end{equation*}
It is easy to know if $\theta>b$, then $x > y$. If $\theta\leq b$, then $y$ is the largest eigenvalue. In this case, the optimal $\theta^*$ is given as follows
 \begin{equation}\label{theta3}
	\theta^*=
	\begin{cases}
		\mu,\quad & \text{if}\;a< \mu \leq b \\
            e_2,\quad & \text{if}\;b< \mu \leq c
	\end{cases}.
\end{equation}

By combining Equations \eqref{theta1}, \eqref{theta2}, and \eqref{theta3}, we have addressed all scenarios where $\mu<a$, $a\leq \mu \leq c$, and $\mu>c$. As such, we derive the optimal value 
of $\theta^*$ for ADMM as follows
\begin{equation*}
	\theta^*=
	\begin{cases}
		\sqrt{\mu a},\quad & \text{if}\;  \mu \leq 2b- a\\
		\sqrt{\frac{\mu a b}{\mu+a-b}},\quad & \text{if}\;  2b - a <\mu \leq a \\
		\mu,\quad & \text{if}\; a < \mu \leq b\\
		\sqrt{\frac{\mu c b}{\mu+c-b}},\quad & \text{if} \;b<\mu	
	\end{cases},
\end{equation*}
which verifies the validity of the closed-form solution in Theorem 4.


\bibliography{aaai24}

\end{document}